\documentclass[eqthmnum,nocolour]{jt-calcs}
\usepackage[backend=bibtex,style=alphabetic,sorting=nyt]{biblatex}
\bibliography{principal-2-block-bib}

\usepackage{mathrsfs}
\usepackage{tikz-cd}
\usepackage{mathscinet}

\usepackage{tikz}
\usetikzlibrary{matrix}

\setlist[enumerate,1]{label=(\alph*)}

\newcommand{\aut}{\mathrm{Aut}}
\newcommand{\syl}{\mathrm{Syl}}
\newcommand{\Syl}{\mathrm{Syl}}
\newcommand{\irr}{\mathrm{Irr}}

\newcommand{\tw}[1]{{}^#1}

\addressm{MSU Denver, PO Box 173362, Campus Box 38, Denver, CO 80217-3362, USA}
\emailm{aschaef6@msudenver.edu}

\addresst{University of Arizona, 617 N. Santa Rita Ave., Tucson AZ 85721, USA}
\emailt{jaytaylor@math.arizona.edu}

\title{Principal $2$-Blocks and Sylow $2$-Subgroups}
\author{A. A.\ Schaeffer Fry and Jay Taylor}

\mscno{2010}{20C15}{20C33}
\keywords{McKay Conjecture, principal 2-block, odd-degree characters}

\begin{document}
\begin{abstract}
Let $G$ be a finite group with Sylow $2$-subgroup $P \leqslant G$. Navarro--Tiep--Vallejo have conjectured that the principal $2$-block of $N_G(P)$ contains exactly one irreducible Brauer character if and only if all odd-degree ordinary irreducible characters in the principal $2$-block of $G$ are fixed by a certain Galois automorphism $\sigma \in \Gal(\mathbb{Q}_{|G|}/\mathbb{Q})$. Recent work of Navarro--Vallejo has reduced this conjecture to a problem about finite simple groups. We show that their conjecture holds for all finite simple groups, thus establishing the conjecture for all finite groups.
\end{abstract}

\section{Introduction}
\begin{pa}
Let $G$ be a finite group, $\ell>0$ a prime, and let $\Irr_{\ell'}(G) \subseteq \Irr(G)$ be the ordinary irreducible characters of $G$ whose degrees are coprime to $\ell$. The McKay conjecture proposes that there is a bijection between the sets $\Irr_{\ell'}(G)$ and $\Irr_{\ell'}(N_G(P))$ where $N_G(P)$ is the normaliser of a Sylow $\ell$-subgroup $P \in \Syl_{\ell}(G)$. In \cite{navarro:2004:the-mckay-conjecture-and-galois} Navarro proposed a striking generalisation of this conjecture, which we refer to as the Galois--McKay conjecture. This conjecture states that there exists a bijection $\Irr_{\ell'}(G) \to \Irr_{\ell'}(N_G(P))$ which is compatible with the action of certain Galois automorphisms.
\end{pa}

\begin{pa}
Currently very little is known about the validity of the Galois-McKay conjecture. However, there is a known consequence of this conjecture which is much more tractable than the conjecture itself, see \cite[5.2]{navarro:2004:the-mckay-conjecture-and-galois}. From now until the end of this article we denote by $\sigma \in \Gal(\mathbb{Q}^{\mathrm{ab}}/\mathbb{Q})$ the unique element of the Galois group of the maximal abelian extension $\mathbb{Q} \subseteq \mathbb{Q}^{\mathrm{ab}}$ that fixes $2$-roots of unity and squares odd roots of unity.
\end{pa}

\begin{conj}[Navarro]\label{conj:SNS2}
Assume $G$ is a finite group and $P \in \Syl_2(G)$. We have $N_G(P) = P$ if and only if all odd-degree irreducible characters of $G$ are $\sigma$-fixed.
\end{conj}

\begin{pa}
The first author has reduced \cref{conj:SNS2} to showing that each finite simple group is SN2S-Good, in the sense of \cite[Definition 1]{schaefferfry:2016:odd-degree-characters-and-self-normalizing-sylow-subgroups}. Moreover, the combined efforts of \cite{schaefferfry:2016:odd-degree-characters-and-self-normalizing-sylow-subgroups,schaeffer-fry-taylor:2017:on-self-normalising-sylow-2-subgroups,schaeffer-fry:2017:galois-automorphisms-harish-chandra} complete the programme of showing each finite simple group is SN2S-good, thus establishing \cref{conj:SNS2}. Recently Navarro--Tiep--Vallejo \cite{navarro-tiep-vallejo:2016:local-blocks-with-one-simple-module} have considered an analogue of \cref{conj:SNS2} which involves the principal $\ell$-block. They show their analogue holds when $\ell$ is odd but, again, the $\ell=2$ case seems to be harder. Their conjecture in the case $\ell=2$, which is the focus of this article, is as follows.
\end{pa}

\begin{conj}[Navarro--Tiep--Vallejo]\label{conj:main}
Assume $G$ is a finite group and $P \in \Syl_2(G)$. The principal $2$-block of $N_G(P)$ contains only one irreducible Brauer character if and only if every odd-degree character in the principal $2$-block of $G$ is $\sigma$-fixed.
\end{conj}

\begin{pa}
In this form \cref{conj:main} and \cref{conj:SNS2} do not appear to be related. However, it is a result of Brauer that the principal $\ell$-block of a finite group has only one irreducible Brauer character if and only if the group has a normal $\ell$-complement, see \cite[Corollary 6.13]{navarro:1998:characters-and-blocks}. In fact, it is shown in \cite[6.7]{navarro-vallejo:2017:2-Brauer-correspondent-blocks} that \cref{conj:SNS2} is a consequence of \cref{conj:main} and \cite[Theorem C]{navarro-vallejo:2017:2-Brauer-correspondent-blocks}. In \cite[Theorem B]{navarro-vallejo:2017:2-Brauer-correspondent-blocks} Navarro--Vallejo have shown that to establish \cref{conj:main} for all finite simple groups it is enough to establish the conjecture when $G$ is an almost simple group whose socle is a finite nonabelian simple group of $2$-power index. Using this approach we are able to establish the validity of \cref{conj:main}.
\end{pa}

\begin{thm}
If $G$ is an almost simple group whose socle is nonabelian and has $2$-power index, then \cref{conj:main} holds for $G$. In particular, \cref{conj:main} holds for all finite groups.
\end{thm}

\begin{pa}\label{pa:exceptions}
Now let $S$ be a finite nonabelian simple group and $S \leqslant A \leqslant \Aut(S)$ an almost simple group with $A/S$ a $2$-group. As one might expect, checking that $A$ satisfies \cref{conj:main} is closely related to checking that $S$ is SN2S-good. Hence, we first consider when our previous work \cite{schaefferfry:2016:odd-degree-characters-and-self-normalizing-sylow-subgroups,schaeffer-fry-taylor:2017:on-self-normalising-sylow-2-subgroups,schaeffer-fry:2017:galois-automorphisms-harish-chandra} establishes that \cref{conj:main} holds for $A$. This turns out to be the case unless $S$ is one of the following groups: a group of Lie type defined in characteristic $2$, $\A_{n-1}^{\pm}(q)$, $\E_6^{\pm}(q)$, ${}^2\G_2(q)$, $\mathsf{J}_1$, $\mathsf{J}_2$, $\mathsf{J}_3$, $\mathsf{Suz}$, $\mathsf{HN}$. Therefore, these are the cases that we must consider here.
\end{pa}

\begin{pa}
The layout of this paper is as follows. In \Cref{sec:normalisers,sec:passing-almost-simple} we recall some general statements about normalisers of Sylow $2$-subgroups and characters of principal $2$-blocks. This allows us to establish when \cref{conj:main} is a consequence of being SN2S-Good, as mentioned in \cref{pa:exceptions}, see \cref{prop:reducetoSN2Sgoodness}. The sporadic groups mentioned in \cref{pa:exceptions} and ${}^2\G_2(q)$ are treated in \cref{sec:sporadics}. In \cref{sec:red-grp} we introduce finite reductive groups and give a criterion for an irreducible character of a finite reductive group to be $\sigma$-fixed. Using this and results of \cite{schaeffer-fry-taylor:2017:on-self-normalising-sylow-2-subgroups,navarro-tiep:2015:irreducible-representations-of-odd-degree} we treat the remaining exceptions from \cref{pa:exceptions} in \cref{sec:char-2,sec:typeA,sec:typeE}.
\end{pa}

\section{Normalisers of Sylow \texorpdfstring{$2$}{2}-Subgroups}\label{sec:normalisers}
\begin{pa}
Assume $S$ is a finite group with a trivial centre and let $S \leqslant A \leqslant \Aut(S)$ be an extension of $S$ whose quotient $A/S$ is a $2$-group. If $Q \in \Syl_2(A)$ then $SQ/S \in \Syl_2(A/S)$ so we must have $SQ/S = A/S$, i.e., $A = SQ$. The intersection $P = S \cap Q \in \Syl_2(S)$ is a Sylow $2$-subgroup of $S$, which is normal in $Q$. We wish to record some elementary lemmas regarding the relationship between $N_A(Q)$ and $N_S(P)$.
\end{pa}

\begin{lem}\label{lem:norm-2-comp-exist}
If $N_S(P) = P \times V$ has a normal $2$-complement $V$, then $N_A(Q) = Q\times C_V(Q)$. Hence, $N_A(Q)$ also has a normal $2$-complement. In particular, if $S$ has a self-normalising Sylow $2$-subgroup, then so does $G$.
\end{lem}

\begin{proof}
By assumption we have $N_S(P) = P \times V$ with $V \leqslant N_S(P)$ a $2'$-group. An easy calculation shows that $N_A(Q) = N_S(Q)Q$. Moreover, as $N_S(Q) \leqslant N_S(P)$ we have $N_S(Q) = P \times N_V(Q)$ so $N_A(Q) = Q \rtimes  N_V(Q)$. Note that $C_S(P) = Z(P) \times V$ and as $Q$ normalises both $S$ and $P$ it must normalise $V$. Hence $Q$ normalises $N_V(Q)$ so $N_A(Q) = Q \times C_V(Q)$ as desired.
\end{proof}

\begin{lem}\label{lem:norm-2-comp-no-exist}
Assume $C_S(P) = Z(P)$ and $N_S(P) = P \rtimes V$ with $V \leqslant S$ a $2'$-group. Then either $N_A(Q) = Q$ or $N_A(Q)$ has no normal $2$-complement.
\end{lem}

\begin{proof}
Assume that $N_A(Q) = Q \times V'$ with $V' \leqslant A$ a $2'$-group.  Then as $V' \leqslant C_A(Q)$ we must have $V' \leqslant C_A(P) \leqslant N_A(P)$. An easy calculation shows that $N_A(P) = N_S(P)Q = QV$. As $Q$ is a $2$-group and $V$ is a $2'$-group we must have $Q \cap V = \{1\}$ so any $2'$-element in $N_A(P)$ is contained in $V$. As $V'$ is a $2'$-group this implies $V' \leqslant C_V(P) \leqslant C_S(P)$ but this implies $V' = \{1\}$ because $C_S(P) = Z(P)$ is a $2$-group by assumption.
\end{proof}

\section{Passing From Almost Simple to Simple Groups}\label{sec:passing-almost-simple}
\begin{assumption}
If $G$ is a finite group then we denote by $B_0(G)$ the principal $2$-block of $G$. Moreover, we denote by $\Irr(B_0(G))$ the ordinary irreducible characters of $G$ contained in the block and by $\Irr_{2'}(B_0(G)) = \Irr(B_0(G)) \cap \Irr_{2'}(G)$ those which have odd degree.
\end{assumption}

\begin{pa}
The main result of \cite{navarro-vallejo:2017:2-Brauer-correspondent-blocks} states that \cref{conj:main} holds for all finite groups if it holds for all almost simple groups $A$ whose quotient $A/S$ by its nonabelian socle $S$ is a $2$-group. In this section we develop several lemmas which allow us to deduce, in certain scenarios, that \cref{conj:main} holds for $A$ if it holds for $S$. \Cref{lem:norm-2-comp-exist} already goes in this direction given the following result of Brauer, see \cite[Corollary 6.13]{navarro:1998:characters-and-blocks}.
\end{pa}

\begin{lem}[Brauer]\label{lem:Brauer}
Let $G$ be a finite group and $P \in \Syl_2(G)$.  Then the principal block $B_0(G)$ contains only one irreducible Brauer character if and only if $N_G(P)$ has a normal $2$-complement.
\end{lem}

\begin{pa}
Let $G$ be a finite group. We now turn our attention to the statement that every character $\chi \in \Irr_{2'}(B_0(G))$ is $\sigma$-fixed. For a group $X$ acting on the irreducible characters $\irr(G)$ we write $\irr_{2'}(G)_X$ for the members of $\irr_{2'}(G)$ that are invariant under $X$.  
\end{pa}

\begin{lem}\label{lem:2powerrestriction}
Assume $G$ is a finite group with normal subgroup $N\lhd G$ whose quotient $G/N$ is a $2$-group. Then given any odd-degree character $\chi \in \Irr_{2'}(G)$, the restriction $\Res_N^G(\chi) \in \Irr_{2'}(N)_{G/N}$ is irreducible. Furthermore, $\chi\in\irr_{2'}(B_0(G))$ if and only if $\Res^{G}_N(\chi)\in \irr_{2'}(B_0(N))_{G/N}$ and $\chi$ is $\sigma$-fixed if and only if $\Res^{G}_N(\chi)$ is $\sigma$-fixed.
\end{lem}

\begin{proof}
Let $\chi\in \irr_{2'}(G)$ and let $\varphi\in\irr(N)$ satisfy $\langle \Res_N^G(\chi), \varphi\rangle\neq 0$.  By Clifford theory, $\chi(1)/\varphi(1)$ divides the index $[G:N]$ which is a $2$-power.  Since $\chi(1)$ is odd it follows that $\chi(1)=\varphi(1)$ and $\Res_N^G(\chi)=\varphi$ is irreducible and invariant under $G/N$. The second statement follows from the observation that the principal block $B_0(G)$ is the only block of $G$ that covers $B_0(N)$, see \cite[Lemma 5.1]{navarro-tiep-vallejo:2016:local-blocks-with-one-simple-module}.  The last statement follows from \cite[Lemma 3.4]{schaefferfry:2016:odd-degree-characters-and-self-normalizing-sylow-subgroups}.
\end{proof}

\begin{cor}\label{conj:ifstatement}
Let $A = SQ$ be an almost simple group with socle $S$ and $Q \in \Syl_2(A)$. Assume $N_A(Q)$ has a normal $2$-complement. If every $Q$-invariant $\chi\in\irr_{2'}(B_0(S))$ is $\sigma$-fixed, then $A$ satisfies \cref{conj:main}.
\end{cor}

\begin{proof}
This is an immediate consequence of \cref{lem:Brauer,lem:2powerrestriction}.
\end{proof}

\begin{cor}\label{rem:Pnormal2complement}
Let $S$ be a finite simple group such that $N_S(P)$ has a normal $2$-complement for some $P \in \Syl_2(S)$. Assume every $\chi\in\Irr_{2'}(B_0(S))$ is $\sigma$-fixed, i.e., $S$ satisfies \cref{conj:main}. Then any almost simple group $S \leqslant A \leqslant \Aut(S)$ with $A/S$ a $2$-group satisfies \cref{conj:main}.
\end{cor}

\begin{proof}
If $Q \in \Syl_2(A)$ then by \cref{lem:norm-2-comp-exist} $N_A(Q)$ has a normal $2$-complement because $N_S(P)$ does. Moreover, by \cref{lem:2powerrestriction} we have every member of $\Irr_{2'}(B_0(A))$ is fixed by $\sigma$ because every member of  $\Irr_{2'}(B_0(S))$ is.
\end{proof}

\begin{pa}
It is known that all finite simple groups are SN2S-Good, see \cite{schaefferfry:2016:odd-degree-characters-and-self-normalizing-sylow-subgroups,schaeffer-fry-taylor:2017:on-self-normalising-sylow-2-subgroups,schaeffer-fry:2017:galois-automorphisms-harish-chandra}. We would like to use these previous results to deduce that if a finite simple group $S$ is SN2S-Good, then any almost simple group $A$ with socle $S$ and quotient $A/S$ a $2$-group satisfies \cref{conj:main}. Unfortunately there are exceptions to this, but we deal with all such cases in the following sections. Before proceeding we will need the following lemma.
\end{pa}

\begin{lem}\label{lem:principal-block-G/N}
Assume $G$ is a finite group and $N \lhd G$ is a normal subgroup. If we identify $\Irr(G/N)$ as a subset of $\Irr(G)$ then we have $\Irr(B_0(G/N)) \subseteq \Irr(B_0(G))$. Moreover, if $N$ is a $2$-group and $G/C_G(N)$ is a $2$-group then we have
\begin{equation*}
\Irr(B_0(G/N)) = \{\chi \in \Irr(B_0(G)) \mid N \leqslant \Ker(\chi)\}.
\end{equation*}
\end{lem}

\begin{proof}
This is just \cite[7.6]{navarro:1998:characters-and-blocks} together with the observation that the trivial character of $G/N$ lifts to the trivial character of $G$.
\end{proof}

\begin{prop}\label{prop:reducetoSN2Sgoodness}
Let $S$ be a nonabelian finite simple group and let $S \leqslant A \leqslant \Aut(S)$ be an almost simple group with $A/S$ a $2$-group. If $S$ is SN2S-Good, in the sense of \cite[Definition 1]{schaefferfry:2016:odd-degree-characters-and-self-normalizing-sylow-subgroups}, then $A$ satisfies \cref{conj:main} unless $S$ is one of the following finite simple groups:
\begin{itemize}
\item $\tw{2}{\G_2(q)}$, $\mathsf{J}_1$, $\mathsf{J}_2$, $\mathsf{J}_3$, $\mathsf{Suz}$, $\mathsf{HN}$,
\item a simple group of Lie type defined in characteristic $2$ whose split maximal torus is nontrivial,
\item $\A_{n-1}^\pm(q)$ with $q$ odd,
\item $\E_6^\pm(q)$ with $q$ odd.
\end{itemize}
\end{prop}

\begin{proof}
First suppose that $S$ has a self-normalising Sylow $2$-subgroup.  Then $A$ has a self-normalising Sylow $2$-subgroup for any $A=SQ$ with $Q\in\syl_2(A)$, see \cref{lem:norm-2-comp-exist}.  Hence, by \cref{conj:ifstatement}, we have $A$ satisfies \cref{conj:main} if every $\chi\in\irr_{2'}(B_0(S))$ is $\sigma$-fixed.  However, if $S$ is SN2S-Good then all odd-degree irreducible characters of $S$ are $\sigma$-fixed.

If $S$ does not have a self-normalising Sylow $2$-subgroup and is not in the stated list of exceptions then, by \cite{kondratiev:2004:normalizers-of-sylow-2-subgroups-in-finite-simple-groups}, $S$ must be $\PSp_{2m}(q)$ with $q\equiv \pm3\pmod 8$ and $m\geq 2$.  

Now, note that since $q\equiv\pm3\pmod8$, $q$ is an odd power of an odd prime so $S$ has index at most $2$ in $A$.  Specifically, if $A\neq S$ then $A=\mathrm{InnDiag}(S)$.  Furthermore, $\mathrm{InnDiag}(S)$ has a self-normalising Sylow $2$-subgroup by \cite[Lemma 3.17]{schaeffer-fry:2017:galois-automorphisms-harish-chandra}.  By \cite[Lemma 7.5]{malle-spaeth:2016:characters-of-odd-degree} and \cite[Lemma 21.14]{cabanes-enguehard:2004:representation-theory-of-finite-reductive-groups}, we see that $\irr_{2'}(G)=\irr_{2'}(B_0(G))$, where $G=\Sp_{2n}(q)$ is the Schur cover for $S$. Thus, by \cref{lem:principal-block-G/N}, we have $\irr_{2'}(S)=\irr_{2'}(B_0(S))$ since $S=G/Z(G)$ with $|Z(G)|=2$. Using \cref{conj:ifstatement} we see that it suffices to show that there exists a character $\chi\in\irr_{2'}(S)$ which is not $\sigma$-fixed but that every $\mathrm{InnDiag}(S)$-invariant character $\chi \in \irr_{2'}(S)$ is $\sigma$-fixed. That is, it suffices to show that $S$ is SN2S-good.
\end{proof}

\section{The Groups \texorpdfstring{$\tw{2}{\G_2(q)}$, $\mathsf{J}_1$, $\mathsf{J}_2$, $\mathsf{J}_3$, $\mathsf{Suz}$, $\mathsf{HN}$}{2G2(q), J1, J2, J3, Suz, HN}}\label{sec:sporadics}

\begin{prop}
If $S$ is one of the simple groups $\tw{2}{\G_2(q)}$, $\mathsf{J}_1$, $\mathsf{J}_2$, $\mathsf{J}_3$, $\mathsf{Suz}$, or $\mathsf{HN}$ and $S \leqslant A \leqslant \Aut(S)$ is an almost simple group with $A/S$ a $2$-group, then $A$ satisfies \cref{conj:main}.
\end{prop}

\begin{proof}
If $S$ is either $\tw{2}{\G_2(q)}$ or $\mathsf{J}_1$ then $S$ does not have a self-normalising Sylow $2$-subgroup and the outer automorphism group has odd order or is trivial, respectively.  In the other cases the outer automorphism group is order $2$ but $\aut(S)$ has a self-normalising Sylow $2$-subgroup.  Hence by \cref{conj:ifstatement}, and  \cite[Theorems 4.2 and 4.3]{schaefferfry:2016:odd-degree-characters-and-self-normalizing-sylow-subgroups}, it suffices to show that the odd degree characters illustrated in the proofs of \cite[Theorems 4.2 and 4.3]{schaefferfry:2016:odd-degree-characters-and-self-normalizing-sylow-subgroups} which are not fixed by $\sigma$ also lie in the principal block. 

In the case of the Ree groups $\tw{2}{\G_2(q)}$ the character $\chi_4$ in the notation of \cite{geck-hiss-lubeck-malle-pfeiffer:1996:chevie}, mentioned in \cite[Theorem 4.3]{schaefferfry:2016:odd-degree-characters-and-self-normalizing-sylow-subgroups} to not be $\sigma$-invariant, is the character $\xi_7$ in the notation of \cite{landrock-michler:1980:principal-2-blocks-of-simple-groups-of-ree-type}, which is shown there to be in the principal block.

The GAP Character Table Library \cite{Breuer:2004:ctbllib-a-gap-package} contains the character tables, Brauer tables, and decomposition matrices for $\mathsf{J}_1, \mathsf{J}_2, \mathsf{J}_3, \mathsf{Suz},$ and $\mathsf{HN}$ in characteristic $2$.  These yield that there are two characters of degree 77, 21, 85, 5005, and 133, respectively, which are interchanged by the action of $\sigma$ and lie in the principal block.
\end{proof}

\section{Generalities on Reductive Groups}\label{sec:red-grp}
\begin{assumption}
From this point forward we assume $p > 0$ is a fixed prime and $\mathbb{K} = \overline{\mathbb{F}}_p$ is an algebraic closure of the finite field $\mathbb{F}_p$ of cardinality $p$.
\end{assumption}

\begin{pa}
Let $\bG$ be a connected reductive algebraic group over $\mathbb{K}$ and let $F : \bG \to \bG$ be a Frobenius endomorphism endowing $\bG$ with an $\mathbb{F}_q$-rational structure $G = \bG^F = \{g \in \bG \mid F(g) = g\}$. Given such a pair $(\bG,F)$ we will denote by $\mathcal{S}(\bG,F)$ the set of all pairs $(\bT,s)$ consisting of an $F$-stable maximal torus $\bT \leqslant \bG$ and a rational semisimple element $s \in \bT^F$. We assume chosen a regular embedding $\iota : \bG \to \widetilde{\bG}$, where $\widetilde{\bG}$ is a connected reductive algebraic group over $\mathbb{K}$ with connected centre. The Frobenius endomorphism on $\widetilde{\bG}$ will also be denoted by $F$ and the group $\widetilde{G}$ denotes the finite group $\widetilde{\bG}^F$.
\end{pa}

\begin{pa}
Let $(\bG^{\star},F^{\star})$ be a pair dual to $(\bG,F)$ and similarly let $(\widetilde{\bG}^{\star},F^{\star})$ be a pair dual to $(\widetilde{\bG},F)$. We will assume that $\iota^{\star} : \widetilde{\bG}^{\star} \to \bG^{\star}$ is a surjective homomorphism of algebraic groups dual to the regular embedding; note this is defined over $\mathbb{F}_q$. To each pair $(\bT^{\star},s) \in \mathcal{S}(\bG^{\star},F^{\star})$ we have a corresponding virtual character $R_{\bT^{\star}}^{\bG}(s)$ of $G=\bG^F$. If $[s] \subseteq G^{\star} := \bG^{\star F^{\star}}$ is a $\bG^{\star F^{\star}}$-conjugacy class of semisimple elements then we have a corresponding (rational) Lusztig series $\mathcal{E}(G,[s]) \subseteq \Irr(G)$. These series form a partition
\begin{equation*}
\Irr(G) = \bigsqcup_{[s] \subseteq G^{\star}} \mathcal{E}(G,[s])
\end{equation*}
of the irreducible characters. We will need the following well-known lemma concerning restriction of characters from $\widetilde{G}$ to $G$, see \cite[11.7]{bonnafe:2006:sln}.
\end{pa}

\begin{lem}\label{lem:res-preimage}
Assume $\chi \in \mathcal{E}(G,[s])$ is an irreducible character and $\widetilde{\chi} \in \Irr(\widetilde{G})$ is an irreducible character covering $\chi$. Then $\widetilde{\chi} \in \mathcal{E}(\widetilde{G},[\widetilde{s}])$ with $\widetilde{s} \in \widetilde{G}^{\star}$ satisfying $\iota^{\star}(\widetilde{s}) = s$.
\end{lem}

\begin{pa}
For the rest of this section we assume that $\gamma \in \Gal(\mathbb{Q}_{|\widetilde{G}|}/\mathbb{Q})$ where $\mathbb{Q}_{|\widetilde{G}|}$ is the field obtained from $\mathbb{Q}$ by adjoining a primitive $|\widetilde{G}|$th root of unity. Moreover, we assume that $\mathcal{E}(G,[s])$ is a $\gamma$-invariant Lusztig series and $\chi \in \mathcal{E}(G,[s])$ is an irreducible character. Let $\widetilde{\chi} \in \Irr(\widetilde{G})$ be a character covering $\chi$ so that $\widetilde{\chi} \in \mathcal{E}(\widetilde{G},\tilde{s})$ with $\tilde{s} \in \widetilde{G}^{\star}$ satisfying $\iota^{\star}(\tilde{s}) = s$ by \cref{lem:res-preimage}. The proof of \cite[3.4]{schaeffer-fry-taylor:2017:on-self-normalising-sylow-2-subgroups} shows that $\mathcal{E}(\widetilde{G},[\tilde{s}])^{\gamma} = \mathcal{E}(\widetilde{G},[\tilde{t}])$ for some semisimple element $\tilde{t} \in \widetilde{G}^{\star}$. Now clearly $\widetilde{\chi}^{\gamma} \in \mathcal{E}(\widetilde{G},[\tilde{t}])$ and $\chi^{\gamma}$ is covered by $\widetilde{\chi}^{\gamma}$ so another application of \cref{lem:res-preimage} shows that $\iota^{\star}(\tilde{t}) = s$ because $\chi^{\gamma} \in \mathcal{E}(G,[s])$ by assumption.
\end{pa}

\begin{pa}
The kernel $\Ker(\iota^{\star})$ is connected, see \cite[2.5]{bonnafe:2006:sln}, so the Lang--Steinberg theorem shows that there exists an element $\tilde{z} \in \Ker(\iota^{\star})^{F^{\star}}$ such that $\tilde{t} = \tilde{s}\tilde{z}$. By \cite[2.6, 11.6]{bonnafe:2006:sln} we have a bijection
\begin{align*}
\mathcal{E}(\widetilde{G},[\tilde{s}]) &\to \mathcal{E}(\widetilde{G},[\tilde{s}\tilde{z}])\\
\widetilde{\chi} &\mapsto \widetilde{\chi}\otimes \theta_{\tilde{z}}
\end{align*}
where $\theta_{\tilde{z}} \in \Irr(\widetilde{G})$ is the lift of a character of the quotient $\widetilde{G}/G$. With this we can prove the following.
\end{pa}

\begin{prop}\label{prop:extend-fixed}
Let $\gamma \in \Gal(\mathbb{Q}_{|\widetilde{G}|}/\mathbb{Q})$ be a Galois automorphism and $\mathcal{E}(G,[s])$ a $\gamma$-invariant Lusztig series. Assume $\tilde{s} \in \widetilde{G}^{\star}$ is such that $\iota^{\star}(\tilde{s}) = s$ and $\widetilde{\chi} \in \mathcal{E}(\widetilde{G},[\tilde{s}])$ satisfies the following property:
\begin{itemize}
	\item[($\star$)] for any $\widetilde{\chi}' \in \mathcal{E}(\widetilde{G},[\tilde{s}])$ we have $\langle \widetilde{\chi}',R_{\widetilde{\bT}^{\star}}^{\widetilde{\bG}}(\tilde{s})\rangle_{\widetilde{G}} = \langle \widetilde{\chi},R_{\widetilde{\bT}^{\star}}^{\widetilde{\bG}}(\tilde{s})\rangle_{\widetilde{G}}$ for all $(\widetilde{\bT}^{\star},\tilde{s}) \in \mathcal{S}(\widetilde{\bG}^{\star},F^{\star})$ if and only if $\widetilde{\chi} = \widetilde{\chi}'$.
\end{itemize}
Then if $\chi \in \mathcal{E}(G,s)$ is a constituent of $\Res_{G}^{\widetilde{G}}(\widetilde{\chi})$, so is $\chi^{\gamma}$. In particular, if $\chi$ extends to $\widetilde{G}$ then $\chi^{\gamma} = \chi$.
\end{prop}

\begin{proof}
The proof of \cite[3.4]{schaeffer-fry-taylor:2017:on-self-normalising-sylow-2-subgroups} together with \cite[11.5(b)]{bonnafe:2006:sln} shows that
\begin{equation}\label{eq:sigma-RTG}
R_{\widetilde{\bT}^{\star}}^{\widetilde{\bG}}(\tilde{s})^{\gamma} = R_{\widetilde{\bT}^{\star}}^{\widetilde{\bG}}(\tilde{t}) = R_{\widetilde{\bT}^{\star}}^{\widetilde{\bG}}(\tilde{s}\tilde{z}) = R_{\widetilde{\bT}^{\star}}^{\widetilde{\bG}}(\tilde{s}) \otimes \theta_{\tilde{z}}
\end{equation}
for any $(\widetilde{\bT}^{\star},\tilde{s}) \in \mathcal{S}(\widetilde{\bG}^{\star},F^{\star})$. Now, this implies that
\begin{equation*}
\langle \widetilde{\chi},R_{\widetilde{\bT}^{\star}}^{\widetilde{\bG}}(\tilde{s}) \rangle = \langle \widetilde{\chi}^{\gamma},R_{\widetilde{\bT}^{\star}}^{\widetilde{\bG}}(\tilde{s})^{\gamma}\rangle = \langle \widetilde{\chi}^{\gamma}\otimes\theta_{\tilde{z}}^{-1}, R_{\widetilde{\bT}^{\star}}^{\widetilde{\bG}}(\tilde{s})\rangle
\end{equation*}
for any $(\widetilde{\bT}^{\star},\tilde{s}) \in \mathcal{S}(\widetilde{\bG}^{\star},F^{\star})$. By assumption ($\star$) we thus have $\widetilde{\chi}^{\gamma} = \widetilde{\chi} \otimes \theta_{\tilde{z}}$. In particular, we must have $\Res_{G}^{\widetilde{G}}(\widetilde{\chi}^{\sigma}) = \Res_{G}^{\widetilde{G}}(\widetilde{\chi})$, which implies $\chi$ and $\chi^{\sigma}$ are both constituents of $\Res_{G}^{\widetilde{G}}(\widetilde{\chi})$.
\end{proof}

\begin{rem}
It has been shown by Digne--Michel in \cite[6.3]{digne-michel:1990:lusztigs-parametrization} that condition ($\star$) in \cref{prop:extend-fixed} is satisfied in all but a few extreme cases. We also note that \cref{prop:extend-fixed} is also true if $\gamma$ is taken to be an automorphism of $\widetilde{G}$ stabilising $G$.
\end{rem}

\section{Characteristic \texorpdfstring{$2$}{2}}\label{sec:char-2}
\begin{assumption}
In this section we assume $\bG$ is simple and simply connected and $p = 2$.
\end{assumption}

\begin{pa}
Let $\bT_0 \leqslant \bB_0 \leqslant \bG$ be a fixed maximal torus and Borel subgroup of $\bG$. We set $\bU_0 = R_u(\bB_0)$ to be the unipotent radical of the Borel. If $\Phi$ are the roots of $\bG$ with respect to $\bT_0$ then for each $\alpha \in \Phi$ we assumed chosen a closed embedding $x_{\alpha} : \mathbb{K}^+ \to \bG$ such that
\begin{equation*}
tx_{\alpha}(c)t^{-1} = x_{\alpha}(\alpha(t)c)
\end{equation*}
for all $t \in \bT_0$ and $c \in \mathbb{K}^+$. If $\Delta \subseteq \Phi$ are the simple roots determined by $\bB_0$ then the set $\{x_{\alpha}(c) \mid \alpha \in \Delta$, $c \in \mathbb{K}^+\}$ generates $\bG$.
\end{pa}

\begin{pa}
Recall that we have a standard Frobenius endomorphism $F_2 : \mathbb{K} \to \mathbb{K}$ given by $F_2(c) = c^2$. We assume fixed a Frobenius endomorphism $F_2 : \bG \to \bG$ such that $F_2(t) = t^2$ for all $t \in \bT_0$. After possibly composing with an inner automorphism of $\bG$ we may (and will) assume that $F_2\circ x_{\alpha} = x_{\alpha}\circ F_2$ for all $\alpha \in \Phi$. Note the groups $\bU_0$ and $\bB_0$ are $F_2$-stable. More generally, if $r = 2^a$ with $a \geqslant 1$ we denote by $F_r$ the $a$-fold composition $F_2\circ\dots\circ F_2$. A graph automorphism of $\bG$ is a finite order automorphism $\tau : \bG \to \bG$ such that $\bT_0$ and $\bB_0$ are $\tau$-stable and $\tau \circ x_{\alpha} = x_{b(\alpha)}$ for some bijection $b : \Delta \to \Delta$. Note that such an automorphism commutes with $F_2$. We will assume that $F = F_q\circ\tau = \tau\circ F_q$ for some (possibly trivial) graph automorphism $\tau$ and some integral power $q$ of $2$. Note $\bT_0$, $\bU_0$, and $\bB_0$ are $F$-stable and we set $T_0 = \bT_0^F$, $U_0 = \bU_0^F$, and $B_0 = \bB_0^F$. We say $F$ is split if $F = F_q$ and twisted otherwise.
\end{pa}

\begin{pa}
If the fixed point group $G = \bG^F$ is perfect then the quotient $S = G/Z$, where $Z := Z(G)$, is a finite simple group of Lie type in characteristic $2$, see \cite[24.14]{malle-testerman:2011:linear-algebraic-groups}. By \cite[2.5.1, 2.5.14]{gorenstein-lyons-solomon:1998:classification-3} we have
\begin{equation*}
\Aut(S) \cong \Aut(G) \cong \widetilde{G}/Z(\widetilde{G}) \rtimes D
\end{equation*}
where $D = \langle F_2, \tau_0,\tau_1\rangle \leqslant \Aut(G)$ is the subgroup generated by field and graph automorphisms. Here we have $\tau_0$ and $\tau_1$ are (possibly trivial) graph automorphisms such that $\tau_0^2 = \tau_1^3 = 1$. Note that if $F$ is twisted then $\tau_0 = \tau_1$ is the identity. In what follows we will denote by $\widetilde{S}$ the group $\widetilde{G}/Z(\widetilde{G})$ and we will also identify $S$ as a subgroup of $\widetilde{S}$.
\end{pa}

\begin{lem}\label{lem:2-sylow-GF-char-2}
Assume $G$ is perfect then we have $U_0 \leqslant G$ is a Sylow $2$-subgroup, $N_G(U_0) = B_0 = U_0 \rtimes T_0$, and $C_G(U_0) = Z(U_0)Z$. Moreover, $N_G(U_0)$ has no normal $2$-complement unless $F = F_2$, in which case $N_G(U_0) = U_0Z$ or equivalently $T_0 = Z$.
\end{lem}

\begin{proof}
The statements about $U_0$, $N_G(U_0)$, and $C_G(U_0)$ are well known, see \cite[2.3.4(d), 2.6.5]{gorenstein-lyons-solomon:1998:classification-3}. It follows that $N_{G}(U_0)$ has a normal $2$-complement only when $T = Z$. If $F = F_2$ is split then $T_0$, hence also $Z$, is trivial so $N_G(U_0) = U_0Z = U_0$ certainly has a normal $2$-complement.

Now assume $F \neq F_2$. By \cite[3.6.7]{carter:1993:finite-groups-of-lie-type} we have $T \neq \{1\}$ but $Z$ is trivial unless $G$ is of type $\A_{n-1}^{\pm}(q)$ or $\E_6^{\pm}(q)$, see \cite[Table 24.2]{malle-testerman:2011:linear-algebraic-groups}. In the case of $\E_6(q)$ we have $|Z| = \gcd(3,q-1) \leqslant 3 < 3^6 \leqslant (q-1)^6 = |T_0|$. In the case of $\E_6^-(q)$ we have $|Z| = \gcd(3,q+1) \leqslant 3 < 3^2 \leqslant (q-1)^4(q+1)^2 = |T_0|$ so $T_0 \neq Z$. In the case of $\A_{n-1}(q)$ we have $|Z| = \gcd(n,q-1) \leqslant n < 3^{n-1} \leqslant (q-1)^{n-1} = |T_0|$. Hence, in these cases we have $T_0 \neq Z$.

Finally consider the case of $\A_{n-1}^-(q)$. Let us write $n-1$ as $2k+\delta$, where $k = \lfloor (n-1)/2 \rfloor$ and $\delta \in \{0,1\}$. If $q > 2$ then we have $|Z| = \gcd(n,q+1) \leqslant n < 3^{n-1} \leqslant (q-1)^{k+\delta}(q+1)^k = |T_0|$. If $q = 2$ then we have $|Z| = \gcd(n,3) \leqslant 3 \leqslant 3^k \leqslant (q-1)^{k+\delta}(q+1)^k = |T_0|$. Hence $T_0 \neq Z$ unless $k=1$ so $G$ is either $\SU_4(2)$ or $\SU_3(2)$. In the first case we have $|Z| = 1 < 3 = |T_0|$ and in the second case we have $\SU_3(2)$ is solvable.
\end{proof}

\begin{pa}\label{pa:sylow-setup}
As $U_0$ is a Sylow $2$-subgroup of $G$ and $Z$ is a $2'$-group we have $P = U_0Z/Z$ is a Sylow $2$-subgroup of $S = G/Z$. The quotient $\widetilde{G}/G \cong \widetilde{S}/S$ has odd order so $P$ is a Sylow $2$-subgroup of $\widetilde{S}$. It's clear that $P$ is $D$-invariant so if $D_2 \leqslant D$ is a Sylow $2$-subgroup of $D$ then $P\rtimes D_2$ is a Sylow $2$-subgroup of $\widetilde{S}\rtimes D$. Now, assume $S \leqslant A \leqslant \widetilde{S}\rtimes D$ is an almost simple group whose quotient $A/S$ is a $2$-group. Up to conjugacy we can then assume that $A = S\rtimes Q_0$ with $Q_0 \leqslant D_2$ a $2$-subgroup. The group $Q := P\rtimes Q_0 \leqslant S\rtimes Q_0$ is a Sylow $2$-subgroup of $A$ and $P = S \cap Q$.
\end{pa}

\begin{lem}\label{lem:norm-sylow-char-2}
Assume $G$ is perfect so that $S = G/Z$ is a finite non-abelian simple group. If $A = S \rtimes Q_0$ is an almost simple group with $Q_0 \leqslant D$ a $2$-group then $Q := P \rtimes Q_0 \in \Syl_2(G)$ and we have $N_A(Q) = Q$ if and only if $F_2 \in Q_0$. If $N_A(Q) \neq Q$ then $N_A(Q)$ has no normal $2$-complement.
\end{lem}

\begin{rem}
Note that if $F = F_2$ then trivially we have $F_2 \in Q_0$ for any subgroup $Q_0 \leqslant D_2$ because $F_2$ is the identity. Hence, if $F = F_2$ then we have $N_A(Q) = Q$ for any almost simple group $A = S \rtimes Q_0$, which follows from \cref{lem:2-sylow-GF-char-2}.
\end{rem}

\begin{proof}[of \cref{lem:norm-sylow-char-2}]
By \cite[Lemma 2.1]{navarro-tiep-turull:2007:p-rational-characters} we have $N_A(Q)/Q \cong C_{N_G(U_0)/U_0Z}(Q_0)$. As we assume $Q_0 \leqslant D_2$ we have the natural map $T_0/Z \to N_G(U_0)/U_0Z$ is a $Q_0$-equivariant isomorphism, hence $C_{N_G(U_0)/U_0Z}(Q_0) \cong C_{T_0/Z}(Q_0)$. Therefore, we have $N_A(Q) = Q$ is self-normalising if and only if $C_{T_0/Z}(Q_0) = \{1\}$. Furthermore, by \cref{lem:norm-2-comp-no-exist} if $C_{T_0/Z}(Q_0) \neq \{1\}$ then $N_A(Q)$ has no normal $2$-complement.

One direction of the statement is clear. If $tZ \in T_0/Z$ is such that $tZ = F_2(tZ) = t^2Z$ then clearly $t \in Z$ so $(T_0/Z)^{F_2} = \{1\}$. Hence, if $F_2 \in Q_0$ then $C_{T_0/Z}(Q_0) = \{1\}$.

We now prove the other direction. Assume $\widecheck{\Delta} = \{\widecheck{\alpha}_1,\dots,\widecheck{\alpha}_n\}$ are the simple coroots. As $\bG$ is simply connected and $p=2$ we have a bijective homomorphism of algebraic groups $(\mathbb{K}^{\times})^n \to \bT_0$ defined by $(\zeta_1,\dots,\zeta_n) \mapsto \widecheck{\alpha}_1(\zeta_1)\cdots\widecheck{\alpha}_n(\zeta_n)$. Assume $\gamma \in D$ is a graph automorphism of order $e \geqslant 1$. Recall that $\gamma$ permutes the set of simple coroots $\widecheck{\Delta}$. We assume fixed a simple coroot $\widecheck{\alpha} \in \widecheck{\Delta}$ whose $\gamma$-orbit has length $e$.

If $r$ is an integral power of $2$ then we define an embedding $t_{r,\gamma} : \mathbb{K}^{\times} \to \bT_0$ by setting
\begin{equation*}
t_{r,\gamma}(\eta) = \widecheck{\alpha}(\eta) \cdot \gamma(\widecheck{\alpha})(\eta^r)\cdots \gamma^{e-1}(\widecheck{\alpha})(\eta^{r^{e-1}}).
\end{equation*}
We note that if $\eta \in \mathbb{F}_{r^e}$ then $t_{r,\gamma}(\eta)$ is fixed by $F_r\circ\gamma$. Now, $Z \neq \{1\}$ only when $\bG$ is of type $\A_n$ or $\E_6$ because $p = 2$. If $\bG$ is of type $\A_n$ or $\E_6$ then the centre $Z(\bG)$ is described in terms of cocharacters in \cite[1.5.6]{geck-malle:2016:reductive-groups-and-steinberg-maps}. From this description we see that we may choose the cocharacter $\widecheck{\alpha}$ such that $t_{r,\gamma}(\eta) \in Z(\bG)$ if and only if $\eta = 1$ unless $\bG = \SL_3(\mathbb{K})$, $r=2$, and $\gamma$ is non-trivial. We will assume $\widecheck{\alpha}$ is chosen with this property.

We also define an embedding $s_0 : \mathbb{K}^{\times} \to \bT_0$ by setting
\begin{equation*}
s_0(\eta) = \widecheck{\alpha}_1(\eta)\cdots\widecheck{\alpha}_n(\eta).
\end{equation*}
With this embedding we have $\tau_0(s_0(\eta)) = s_0(\eta)$ and $F_2(s_0(\eta)) = s_0(\eta^2)$ so if $\eta \in \mathbb{F}_q$ then $s_0(\eta)$ is fixed by $F_q$. One easily sees using the above mentioned descriptions of $Z(\bG)$ that $s_0(\eta) \in Z(\bG)$ if and only if $\eta = 1$.

Let $d \geqslant 1$ denote the order of $\tau$ and write $\bar{q} = q^d = 2^{2^tm}$ with $t \geqslant 0$ and $m \geqslant 1$ odd. With this notation we have that $D_2 = \langle F_2^m, \tau_0\rangle$. We will deal with the case $G = {}^3\D_4(q)$ first. As $2^{m-1}$ divides $q^d-1$, there exists a non-trivial root of unity $\eta_0 \in \mathbb{F}_{\bar{q}}^{\times}$ such that $\eta_0^{2^m-1} = 1$. The element $t_{q,\tau}(\eta_0)Z \in T_0/Z$ is a non-trivial $D_2$-invariant element. By \cref{lem:2-sylow-GF-char-2} we have $T_0 \neq Z$ so we must have $C_{T_0/Z}(Q_0) \neq \{1\}$.

We assume now that $G \neq {}^3\D_4(q)$ so $d \in \{1,2\}$. Assume $m > 1$. The term $2^m-1$ divides $q-1$ because we must have $q = 2^{2^{t+1-d}m}$. Therefore, there exists a non-trivial root of unity $\eta_0 \in \mathbb{F}_{{q}}^{\times}$ such that $\eta_0^{2^m-1} = 1$. The element $s_0(\eta_0)Z$ is non-trivial and $D_2$-invariant, so we're done in this case.

Now assume $m=1$. If $F$ is twisted then we have $\bar{q} \geqslant 4$ and there exists a non-trivial $3$rd root of unity $\eta_0 \in \mathbb{F}_{\bar{q}}^{\times}$. Moreover, we have $D_2 = \langle F_2\rangle$. The element $t_{q,\tau}(\eta_0)$ is $\langle F_2^2\rangle$-invariant and $t_{q,\tau}(\eta_0) \not\in Z$ by the above remarks because we assume $G \neq \SU_3(2)$. Hence, we have $t_{q,\tau}(\eta_0)Z \in C_{T_0/Z}(Q_0) \neq \{1\}$ assuming $F_2 \not\in Q_0$ because $Q_0 \leqslant \langle F_2^2\rangle$.

Now we assume that $F$ is split so $q = \bar{q}$. Assuming $q > 2$ there exists a non-trivial $3$rd root of unity $\eta_0 \in \mathbb{F}_q^{\times}$ and the element $s_0(\eta_0)$ is $\langle F_2^2,\tau_0\rangle$-invariant. Assume $\tau_0$ is not the identity then $\tau_0$ has order $2$ because $G \neq {}^3\D_4(q)$. The element $t_{2,\tau_0}(\eta_0)$ is $\langle F_2^2,F_2\circ\tau_0\rangle$-invariant. Note that this element is contained in $T_0$ because $F_q(t_{\tau_0}(\eta_0)) = t_{\tau_0}(\eta_0^q) = t_{\tau_0}(\eta_0)$ and, moreover, we have $t_{\tau_0}(\eta_0) \not\in Z$. If $F_2 \not\in Q_0$ then $Q_0$ is a subgroup of $\langle F_2^2,\tau_0\rangle$ or $\langle F_2^2,F_2\circ\tau_0\rangle$ so we must have $C_{T_0/Z}(Q_0) \neq \{1\}$. Finally, if $q = 2$ then $F = F_2$ so $T_0 = Z$, by \cref{lem:2-sylow-GF-char-2}, hence we have $C_{T_0/Z}(Q_0) = \{1\}$. Trivially we have $F_2 \in Q_0$ because $F_2$ is the identity in this case. This completes the proof.
\end{proof}

\begin{prop}
Assume $\bG$ is simple and simply connected of arbitrary type and $p = 2$. If $G$ is perfect, so that $S = G/Z$ is a non-abelian simple group, and $S \leqslant A \leqslant \Aut(S)$ is an almost simple group with $A/S$ a $2$-group then $A$ satisfies \cref{conj:main}.
\end{prop}

\begin{proof}
The group $S$ has a strongly split $BN$-pair, in the sense of \cite[2.20]{cabanes-enguehard:2004:representation-theory-of-finite-reductive-groups}, and satisfies the hypothesis in \cite[6.14]{cabanes-enguehard:2004:representation-theory-of-finite-reductive-groups}. Indeed, $G$ satisfies this hypothesis by \cite[6.15]{cabanes-enguehard:2004:representation-theory-of-finite-reductive-groups} hence so does $S$ because it's a quotient of $G$ by a $2'$-group. According to \cite[6.18]{cabanes-enguehard:2004:representation-theory-of-finite-reductive-groups}, as $C_S(P) = Z(P)$, we have every $2$-block of $S$ is either the principal block or a block of defect zero. In particular, this implies that
\begin{equation}\label{eq:princ-block-char-2}
\Irr_{2'}(S) = \Irr_{2'}(B_0(S))
\end{equation}
because any irreducible character of $S$ with maximal defect is in the principal block.

We will write $A = S \rtimes Q_0$, with $Q_0 \leqslant D_2$ a $2$-group, and set $Q = P \rtimes Q_0$ a Sylow $2$-subgroup of $A$ as in \cref{pa:sylow-setup}. If $F_2 \in Q_0$ then we are in the case that $N_A(Q) = Q$ by \cref{lem:norm-sylow-char-2}. It follows from the proof of \cite[5.8]{schaeffer-fry-taylor:2017:on-self-normalising-sylow-2-subgroups} that every $Q_0$-invariant member of $\Irr_{2'}(S)$ is fixed by $\sigma$, so $A$ satisfies \cref{conj:main} by \cref{conj:ifstatement}. Now assume $F_2 \not\in Q_0$ then $N_A(Q)$ has no normal $2$-complement, c.f., \cref{lem:norm-sylow-char-2}. We must show that there exists an odd degree character $\chi \in \Irr_{2'}(B_0(A))$ which is not $\sigma$-fixed.

It suffices to find an $A$-invariant character $\chi \in \Irr_{2'}(S)$ which is not $\sigma$-fixed and extends to $A$. Indeed, if $\widetilde{\chi} \in \Irr_{2'}(A)$ is such an extension, then $\widetilde{\chi}$ is in the principal block of $A$ and is not $\sigma$-fixed, see \cref{eq:princ-block-char-2,lem:2powerrestriction}. If $q=2$ and $F$ is twisted then $D$ is trivial, so $A = S$ because $\widetilde{S}/S$ is odd. In this case the existence of an odd degree character of $S$ which is not $\sigma$-fixed was shown in \cite{schaefferfry:2016:odd-degree-characters-and-self-normalizing-sylow-subgroups}.

Now we can assume $q > 2$. In the proof of \cite[6.4]{schaeffer-fry-taylor:2017:on-self-normalising-sylow-2-subgroups} it is shown that there exists an odd-degree character $\chi \in \Irr_{2'}(G)$ with the following properties: $\chi$ extends to $\widetilde{G}$,  is $Q_0$-invariant, has $Z$ in its kernel, and is not $\sigma$-fixed. As $\chi$ extends to $\widetilde{G}$, it follows from \cite[3.4]{spaeth:2012:inductive-mckay-defining} that $\chi$ extends to its inertia group $G\rtimes D_{\chi}$ in the semidirect product $G\rtimes D$. Let $\widetilde{\chi} \in \Irr(G\rtimes D_{\chi})$ be such an extension. Then clearly $Z$ is in the kernel of $\widetilde{\chi}$ so we may view this as a character of $S\rtimes D_{\chi}$. As $Q_0 \leqslant D_{\chi}$ we have $\Res_A^{S\rtimes D_{\chi}}(\widetilde{\chi})$ is an extension of $\chi$ so $A$ satisfies \cref{conj:main}.
\end{proof}

\begin{pa}
We end this section by making some remarks on the work in \cite[\S5, \S6]{schaeffer-fry-taylor:2017:on-self-normalising-sylow-2-subgroups} concerning groups in characteristic $2$. For the following remarks we adopt the notation of \cite{schaeffer-fry-taylor:2017:on-self-normalising-sylow-2-subgroups}. Assuming $G$ is perfect the statement of \cite[5.7]{schaeffer-fry-taylor:2017:on-self-normalising-sylow-2-subgroups} is correct but there is insufficient detail in the proof. The statement we require (and use) is that $C_{N_G(P)/PZ}(Q) = \{1\}$ if and only if $F_2 \in \Gamma_Q(G)$ where $Z = Z(G)$. Arguing exactly as in the proof of \cite[5.7]{schaeffer-fry-taylor:2017:on-self-normalising-sylow-2-subgroups} we have $C_{N_G(P)/PZ}(Q) \cong C_{T_0/Z}(\Gamma_Q(G))$. The proof of \cref{lem:norm-sylow-char-2} now shows that $C_{T_0/Z}(\Gamma_Q(G)) = \{1\}$ if and only if $F_2 \in \Gamma_Q(G)$ giving the statement.
\end{pa}

\begin{pa}
In the proof of \cite[6.4]{schaeffer-fry-taylor:2017:on-self-normalising-sylow-2-subgroups} one has to be more careful. There one should really work with $\bar{q}$ instead of $q$ where $\bar{q}$ is defined as in the proof of \cref{lem:norm-sylow-char-2}. Replacing $q$ by $\bar{q}$ the argument given in \cite[6.4]{schaeffer-fry-taylor:2017:on-self-normalising-sylow-2-subgroups} works for twisted groups but needs to be modified for split groups as in the proof of \cref{lem:norm-sylow-char-2}. In any case, our desired statement is covered either by the arguments in \cite[6.4]{schaeffer-fry-taylor:2017:on-self-normalising-sylow-2-subgroups} or \cite{schaefferfry:2016:odd-degree-characters-and-self-normalizing-sylow-subgroups} except when $G$ is $\E_6(q)$. However, this case is easily dealt with using the same arguments. We omit the details.
\end{pa}

\section{Type \texorpdfstring{$\A_{n-1}$}{A(n-1)}}\label{sec:typeA}
\begin{assumption}
In this section we assume $\bG = \SL_n(\mathbb{K})$ and $p \neq 2$.
\end{assumption}

\begin{lem}\label{lem:type-A-char-inv}
If $s \in G^\star$ is a $2$-element, then $\chi^{\sigma} = \chi$ for all $\chi \in \mathcal{E}(G,[s])\cap \Irr_{2'}(G)$ unless $G = \SL_2(q)$ and $q \equiv \pm3 \pmod{8}$. In this latter case if $\chi \in \mathcal{E}(G,[s])\cap \Irr_{2'}(G)$ does not extend to $\widetilde{G}$, then it is not $\sigma$-fixed.
\end{lem}

\begin{proof}
As $s$ is a $2$-element we have $\mathcal{E}(G,s)$ is $\sigma$-invariant by \cite[3.4]{schaeffer-fry-taylor:2017:on-self-normalising-sylow-2-subgroups}. Assume $\chi \in \mathcal{E}(G,[s])$ is an odd-degree irreducible character and let $\widetilde{\chi} \in \Irr(\widetilde{G})$ be a character covering $\chi$. By \cite[6.3]{digne-michel:1990:lusztigs-parametrization} the condition ($\star$) in \cref{prop:extend-fixed} is satisfied so we have $\widetilde{\chi}$ covers $\chi^{\sigma}$. Moreover, if $\chi$ extends to $\widetilde{G}$ then $\chi^{\sigma} = \chi$. It is shown in \cite[10.2]{schaeffer-fry-taylor:2017:on-self-normalising-sylow-2-subgroups} that if $\chi$ does not extend to $\widetilde{G}$ then $n = 2^r$, for some $r \geqslant 1$, and $\Res_{G}^{\widetilde{G}}(\widetilde{\chi}) = \chi + {}^g\chi$ for some $g \in \widetilde{G}$. If $r > 1$ then one can show that $\chi$ and ${}^g\chi$ are $\sigma$-fixed using the argument of \cite[\S10]{schaeffer-fry-taylor:2017:on-self-normalising-sylow-2-subgroups}. The case of $\SL_2(q)$ is easily checked using the character table given in \cite{bonnafe:2011:SL2}.
\end{proof}

\begin{prop}
Assume $\bG = \SL_n(\mathbb{K})$ and $p \neq 2$. If $S = G/Z(G)$ is simple and $S \leqslant A \leqslant \Aut(S)$ is an almost simple group with $A/S$ a $2$-group, then $A$ satisfies \cref{conj:main}.
\end{prop}

\begin{proof}
Let us start with the case where $G = \SL_2(q)$ and $q \equiv \pm3 \pmod{8}$. If $P \leqslant S$ is a Sylow $2$-subgroup then $N_S(P) \cong (C_2 \times C_2) \rtimes  C_3$ with $C_3$ cyclically permuting the elements of order $2$. Thus $N_S(P)$ does not have a normal $2$-complement. There are precisely $2$ odd degree characters of $G$ that do not extend to $\widetilde{G}$ and these are labelled $R_{\sigma}'(\theta_0)$ in \cite[Table 5.4]{bonnafe:2011:SL2}; they are both in the principal block of $G$ by \cite[7.1.1(e)]{bonnafe:2011:SL2}. By \cite[Exercise 4.1(c)]{bonnafe:2011:SL2} we have $R_{\sigma}'(\theta_0)$ has $Z(G) \cong C_2$ in its kernel so $R_{\sigma}'(\theta_0) \in \Irr_{2'}(B_0(S))$ by \cref{lem:principal-block-G/N}. As these are not $\sigma$-fixed, c.f., \cref{lem:type-A-char-inv}, we have \cref{conj:main} holds for $S$.

Assume now that $A \neq S$. Recall that if $q = p^m$ then $\Aut(S) \cong \PGL_2(q)\rtimes C_m$, where the cyclic group acts via field automorphisms. As $p$ is odd and $q \equiv \pm3 \pmod{8}$ we must have $m$ is odd so $A \cong \PGL_2(q)$. We thus have $A$ has a self-normalising Sylow $2$-subgroup, see \cite[Lemma 3]{carter-fong:1964:the-Sylow-2-subgroups}. It follows easily from the character table of $\GL_2(q)$, see \cite[\S15.9]{digne-michel:1991:representations-of-finite-groups-of-lie-type}, that the only odd degree characters of $\PGL_2(q)$ are the trivial and Steinberg characters, which are $\sigma$-fixed; thus \cref{conj:main} holds for $A$.

We now consider the case where either $G \neq \SL_2(q)$ or $q \equiv \pm1 \pmod{8}$. As $\bG$ is of type $\A$ and $q$ is odd we have by \cite[21.14]{cabanes-enguehard:2004:representation-theory-of-finite-reductive-groups} that
\begin{equation}\label{eq:typ-A-principal-2-block}
\Irr(B_0(G)) = \bigcup_{s \in G^{\star}}\mathcal{E}(G,[s])
\end{equation}
where the sum is taken over all $2$-elements in the dual. \Cref{lem:type-A-char-inv} thus implies that every member of $ \Irr_{2'}(B_0(G))$ is fixed by $\sigma$ and by \cref{lem:principal-block-G/N}, the same is true of $\Irr_{2'}(B_0(S))$. According to \cite[\S1, Corollary]{kondratiev:2004:normalizers-of-sylow-2-subgroups-in-finite-simple-groups} the normaliser $N_S(P)$ has a normal $2$-complement, so $A$ satisfies \cref{conj:main} by \cref{rem:Pnormal2complement}.
\end{proof}

\section{Type \texorpdfstring{$\E_6$}{E6}}\label{sec:typeE}

\begin{assumption}
In this section we assume $\bG$ is simply connected of type $\E_6$ and $p \neq 2$.
\end{assumption}

\begin{prop}
Assume $\bG$ is simply connected of type $\E_6$ and $p \neq 2$. If $S = G/Z(G)$ is simple and $S \leqslant A \leqslant \Aut(S)$ is an almost simple group with $A/S$ a $2$-group, then $A$ satisfies \cref{conj:main}.
\end{prop}

\begin{proof}
Write $A=SQ$ where $Q\in\syl_2(A)$.  First, note that if $P\in\syl_2(S)$, then $N_S(P)$ has a normal $2$-complement, see \cite[\S1, Corollary]{kondratiev:2004:normalizers-of-sylow-2-subgroups-in-finite-simple-groups}. Thus it suffices by \cref{rem:Pnormal2complement} to show that every member of $\Irr_{2'}(B_0(S))$ is $\sigma$-fixed.

Consider an adjoint quotient $\widetilde{\bG} \to \widetilde{\bG}_{\ad}$ of $\widetilde{\bG}$. The kernel of this map is the (connected) centre of $\widetilde{\bG}$, so by the Lang--Steinberg theorem we have $\widetilde{G}/Z(\widetilde{G}) \cong \widetilde{\bG}_{\ad}^F$. Now let $\chi\in\irr_{2'}(B_0(S))$ be non-unipotent. By a result of Navarro--Tiep, together with the proceeding remark, $\chi$ extends to a character $\widetilde{\chi} \in \Irr(\widetilde{G}/Z(\widetilde{G}))$, see \cite[Lemma 4.13]{navarro-tiep:2015:irreducible-representations-of-odd-degree}. By inflation, we may view $\chi$ as a character of $G$ and $\widetilde{\chi}$ as a character of $\widetilde{G}$ extending $\chi$. By \cref{lem:principal-block-G/N}, we have $\chi \in \Irr(B_0(G))$, so a result of Brou\'e--Michel shows that $\chi\in\mathcal{E}(G,[s])$ with $s\in G^{\star}$ a $2$-element, see \cite[9.12]{cabanes-enguehard:2004:representation-theory-of-finite-reductive-groups}. By \cite[3.4]{schaeffer-fry-taylor:2017:on-self-normalising-sylow-2-subgroups} the series $\mathcal{E}(G,[s])$ is $\sigma$-stable and condition ($\star$) of \cref{prop:extend-fixed} is satisfied by \cite{digne-michel:1990:lusztigs-parametrization} because $\chi$ is not unipotent. Hence, as $\chi$ extends to $\widetilde{G}$, we have $\chi$ is $\sigma$-fixed by \cref{prop:extend-fixed}.
\end{proof}

\setstretch{0.96}
\renewcommand*{\bibfont}{\small}
\printbibliography
\end{document}